\documentclass[11pt]{amsart}

\usepackage[english]{babel}
\usepackage{amsmath, amsfonts, amssymb, amsthm, epsfig, graphicx, url, hyperref, color, tikz, cancel}

\newtheorem{theorem}{Theorem}[section]

\newtheorem{corollary}{Corollary}[section]
\newtheorem{proposition}{Proposition}[section]

\newtheorem{remark}{Remark}[section]

\newcounter{theor}

\newtheorem{thm}[theor]{Theorem}

\def\s{\mathbb{S}}

\def\R{\mathbb{R}}

\def\vol{\mathrm{vol}}

\def\e{\mathrm{e}}
\newcommand{\dlat}{\mathrm{d}}

\def\e{\mathrm{e}}
\def\esc#1{\left\langle #1\right\rangle}
\def\g{\mathrm{g}}
\newcommand{\centroid}[1]{\mathrm{g}_{#1}}

\numberwithin{equation}{section}

\begin{document}
\title[A general functional version of Gr\"unbaum's inequality]{A general functional version of Gr\"unbaum's inequality}

\author[D. Alonso]{David Alonso-Guti\'errez}
\address{\'Area de an\'alisis matem\'atico, Departamento de matem\'aticas, Facultad de Ciencias, Universidad de Zaragoza, C/ Pedro Cerbuna 12, 50009 Zaragoza (Spain), IUMA}
\email{alonsod@unizar.es}

\author[F. Mar\'in]{Francisco Mar\'in Sola}
\address{Departamento de Matem\'aticas, Universidad de Murcia, Campus de
Espinar\-do, 30100 Murcia, Spain}
\email{francisco.marin7@um.es} \email{jesus.yepes@um.es}

\author[J. Mart\'in]{Javier Mart\'in Go\~ni}
\address{\'Area de an\'alisis matem\'atico, Departamento de matem\'aticas, Facultad de Ciencias, Universidad de Zaragoza, C/ Pedro Cerbuna 12, 50009 Zaragoza (Spain), IUMA and
Faculty of Computer Science and Mathematics, University of Passau, Innstrasse 33,
94032 Passau, Germany}
\email{j.martin@unizar.es; javier.martingoni@uni-passau.de}

\author[J. Yepes]{Jes\'us Yepes Nicol\'as}

\thanks{The first author is supported by MICINN project PID2022-137294NB-I00 and DGA project
E48$\_$23R. The second and the fourth authors are supported by the grant PID2021-124157NB-I00, funded by MCIN/AEI/10.13039/501100011033/``ERDF A way of making Europe'', as well as by the grant  ``Proyecto financiado por la CARM a través de la convocatoria de Ayudas a proyectos para el desarrollo de investigación científica y técnica por grupos competitivos, incluida en el Programa Regional de Fomento de la Investigación Científica y Técnica (Plan de Actuación 2022) de la Fundación Séneca-Agencia de Ciencia y Tecnología de la Región de Murcia, REF. 21899/PI/22''. The third author is supported by the Austrian Science Fund (FWF) Project P32405 \textit{Asymptotic Geometric Analysis and Applications}.
The fourth author is also supported by the grant RYC2021-034858-I funded by MCIN/AEI/10.13039/501100011033 and by the ``European Union NextGenerationEU/PRTR''}



\begin{abstract}
A classical inequality by Gr\"unbaum provides a sharp lower bound for the ratio $\vol(K^{-})/\vol(K)$, where $K^{-}$ denotes the intersection of a convex body with non-empty interior $K\subset\R^n$ with a halfspace bounded by a hyperplane $H$ passing through the centroid $\g(K)$ of $K$.

\smallskip

In this paper we extend this result to the case in which the hyperplane $H$ passes by any of the points lying in a whole uniparametric family of $r$-powered centroids associated to $K$ (depending on a real parameter $r\geq0$), by proving a more general functional result on concave functions.

\smallskip

The latter result further connects (and allows one to recover) various inequalities involving the centroid, such as a classical inequality (due to Minkowski and Radon) that relates the distance of $\g(K)$ to a supporting hyperplane of $K$, or a result for volume sections of convex bodies proven independently by Makai Jr.\,\&\,Martini and Fradelizi.
\end{abstract}

\maketitle

\section{Introduction}

Let $K\subset\R^n$ be a compact set with positive volume $\vol(K)$, i.e., with positive
$n$-dimensional Lebesgue measure (along the paper, the $k$-dimensional Lebesgue measure of $M$, provided that $M$ is measurable, is denoted by $\vol_k(M)$ and we will omit the index $k$ when it is equal to the dimension $n$ of the ambient space; furthermore, when integrating $\dlat x$ will stand for $\dlat \vol(x)$).
The centroid of $K$ is the affine-covariant point
\[\g(K):=\frac{1}{\vol(K)}\int_K x\,\dlat x.\]

According to a classical result by Gr\"unbaum \cite{Gr}, if $K$ is convex (from now on a compact convex set will be referred to as a \emph{convex body}) with centroid at the origin, then
\begin{equation}\label{e:Grunbaum}
\frac{\vol(K^{-})}{\vol(K)}\geq\left(\frac{n}{n+1}\right)^n,
\end{equation}
where $K^{-}=K\cap\{x\in\R^n: \esc{x,u}\leq0\}$ and $K^{+}=K\cap\{x\in\R^n: \esc{x,u}\geq0\}$
represent the parts of $K$ which are split by the (vector) hyperplane $H=\{x\in\R^n: \esc{x,u}=0\}$,
for any given $u\in\s^{n-1}$.

\medskip

There exists a classical inequality similar in spirit to Gr\"unbaum's result, attributed to Minkowski for $n=2,3$ and Radon for general $n$, which bounds the distance from $\g(K)$ to a supporting hyperplane of the convex body $K$ (see \cite[p.~57-58]{BoFe}).
This result asserts that when $K$ has centroid at the origin then $K\subset-nK$, a fact that is equivalent to the following statement (here $M|E$ denotes the orthogonal projection of the subset $M\subset\R^n$ onto a vector subspace $E$ of $\R^n$):
\begin{thm}\label{t:Minkowski_Radon}
Let $K\subset\R^n$ be a convex body with non-empty interior and let $H$ be a hyperplane. If $K$ has centroid at the origin then
\begin{equation}\label{e:Minkowski_Radon}
\frac{\vol_1\bigl(K^-|H^\perp\bigr)}{\vol_1\bigl(K|H^\perp\bigr)}\geq \frac{1}{n+1}.
\end{equation}
\end{thm}

\smallskip

Another inequality of this type, but now involving volume sections instead of projections, is the following inequality \eqref{e:Frad_Makai_Martini}. It was shown (independently) by \cite{MaMaI}, and later by \cite{Fr}, who further proved this result when considering sections by planes of arbitrary dimension.

\begin{thm}[\cite{Fr,MaMaI}]\label{t:Frad_Makai_Martini}
Let $K\subset\R^n$ be a convex body with non-empty interior, let $H$ be a hyperplane and let $f:[a,b]\longrightarrow \R_{\geq 0}$ be the function given by $f(t) = \vol_{n-1}\bigl({K \cap (tu + H)}\bigr)$.
If $K$ has centroid at the origin then
\begin{equation}\label{e:Frad_Makai_Martini}
\frac{f(0)}{\|f\|_{\infty}}\geq \left(\frac{n}{n+1}\right)^{n-1}.
\end{equation}
\end{thm}


\smallskip

Gr\"unbaum's result was extended to the case of sections \cite{FrMeYa,MyStZh} and projections \cite{StZh} of compact convex sets, and generalized to the analytic setting of \emph{log-concave} functions \cite{MeNaRyYa} (see also \cite[Lemma~2.2.6]{BrGiVaVr}) and $p$-concave functions \cite{MyStZh}, for $p>0$. Other Gr\"unbaum type inequalities involving volumes of sections of compact convex sets through their centroid, later generalized to the case of classical and dual \emph{quermassintegrals} in \cite{StYa}, can be found in \cite{Fr,MaMaI}.

\smallskip

Exploiting the original proof of Gr\"unbaum in \cite{Gr}, the following extension of Gr\"unbaum's inequality \eqref{e:Grunbaum} to the case of compact sets with some concavity for the function which gives the volumes of cross-sections parallel to a given hyperplane (having Gr\"unbaum's inequality as the particular case $p=1/(n-1)$) was shown in \cite{MSYN}.

\begin{thm}[\cite{MSYN}]\label{t:Grunbaum_pgeq0}
Let $K\subset\R^n$ be a compact set with non-empty interior and with centroid at the origin. Let $H$ be a hyperplane such that the function $f:H^{\bot}\longrightarrow\R_{\geq0}$ given by $f(x)=\vol_{n-1}\bigl(K\cap(x+H)\bigr)$ is $p$-concave, for some $p\in(0,\infty)$. Then
\begin{equation}\label{e:Grunbaum_p>0}
\frac{\vol(K^{-})}{\vol(K)}\geq\left(\frac{p+1}{2p+1}\right)^{(p+1)/p}.
\end{equation}
The inequality is sharp.
\end{thm}

Following the idea of the proof of the previous result, we will first show that an analogous statement holds true when one replaces the centroid by the midpoint in a direction $u\in\s^{n-1}$, namely, the point $\bigl[(a+b)/2\bigr]\cdot u$, where $[a,b]$ is the support of the function $f:\R\longrightarrow\R_{\geq0}$ given by \[f(t)=\vol_{n-1}\bigl(K\cap(tu+H)\bigr)\]
(see Proposition \ref{p:midpoint}).

In view of these results (Theorem \ref{t:Grunbaum_pgeq0} and Proposition \ref{p:midpoint}), here we ask about the possibility of finding other particular points that ensure a large enough amount of mass in both subsets that are obtained when cutting the given compact set $K\subset\R^n$ by a hyperplane passing through them.
To figure out such a possible family of points we notice that, fixed a unit direction $u\in\s^{n-1}$, the corresponding components w.r.t. $u$ of both the centroid and the midpoint have a similar nature. Indeed, the component of $\g(K)$ w.r.t. $u$ is given by (see \eqref{e:g(K)_first_comp})
\begin{equation*}
[\g(K)]_1=\frac{1}{\vol(K)}\int_a^b tf(t)\,\dlat t=\frac{\int_a^b t f^1(t) \,\dlat t}{\int_a^b f^1(t) \,\dlat t},
\end{equation*}
whereas the corresponding component of the midpoint is
\begin{equation*}
\frac{a+b}{2}=\frac{\int_a^b t f^0(t) \,\dlat t}{\int_a^b f^0(t) \,\dlat t}.
\end{equation*}
Thus, with the above-mentioned aim in mind, it seems reasonable to consider the points $\centroid{r}\cdot u$, where
\begin{equation}\label{e:lambda_r}
\centroid{r}:=\frac{\int_a^b t f^r(t) \,\dlat t}{\int_a^b f^r(t) \,\dlat t}
\end{equation}
for any $r\geq0$. 
Here we show that such a uniparametric class of points allows us to extend Gr\"unbaum's inequality (or more generally Theorem \ref{t:Grunbaum_pgeq0}) to the case in which one replaces the classical centroid by any of them.

\begin{theorem}\label{t:p_concave_r_centroid}
Let $r\in[0,\infty)$ and let $K\subset\R^n$ be a compact set with non-empty interior having the point $\centroid{r}\cdot u$, with respect to some direction $u\in\s^{n-1}$, at the origin. Let $H=\{x \in \R^{n}: \esc{x,u}=0\}$ be the hyperplane with normal vector $u$ and assume that the function $f:H^{\bot}\longrightarrow\R_{\geq0}$ given by $f(x)=\vol_{n-1}\bigl(K\cap(x+H)\bigr)$ is $p$-concave, for some $p\in(0,\infty)$. If $r\geq1$ then
\begin{equation*}
\frac{\vol(K^{-})}{\vol(K)}\geq\left(\frac{p+1}{2p+r}\right)^{(p+1)/p},
\end{equation*}
whereas if $0\leq r\leq1$ then
\begin{equation*}
\frac{\vol(K^{-})}{\vol(K)}\geq\left(\frac{p+r}{2p+r}\right)^{(p+1)/p}.
\end{equation*}
\end{theorem}
Notice that the cases $r=1$ and $r=0$ correspond to Theorem \ref{t:Grunbaum_pgeq0} and Proposition \ref{p:midpoint}, respectively.

\smallskip

Taking into account that, once a unit direction $u\in\s^{n-1}$ is fixed, the above geometric results are reduced to the study of one variable functions with certain concavity, here we deal with the corresponding functional counterpart
of these statements (from which the latter result will be obtained as a consequence of such an equivalent functional one). To this aim, first we need to define the notion of functional $\alpha$-centroid: given a non-negative function $h:[a,b]\longrightarrow[0,\infty)$ with positive integral, for any $\alpha>0$ we will write
\begin{equation}
\centroid{\alpha}(h):=\frac{\int_a^b th^\alpha(t)\,\dlat t}{\int_a^b h^\alpha(t)\,\dlat t}.
\end{equation}
Now, the statement of our main result reads as follows.

\begin{theorem}\label{t:functional_centroid_general_alpha_beta}
Let $h:[a,b]\longrightarrow[0,\infty)$ be a non-negative concave function, and let $\alpha,\beta>0$.
If $\beta\leq\alpha$ then
\begin{equation}\label{e:int_ineq_beta<alpha}
\frac{\int_{\centroid{\alpha}(h)}^bh^\beta(t)\,\dlat t} {\int_{a}^bh^\beta(t)\,\dlat t} \geq\left(\frac{\beta+1}{\alpha+2}\right)^{\beta+1},
\end{equation}
whereas if $\alpha\leq\beta$ then
\begin{equation}\label{e:int_ineq_alpha<beta}
\frac{\int_{\centroid{\alpha}(h)}^bh^\beta(t)\,\dlat t} {\int_{a}^bh^\beta(t)\,\dlat t} \geq\left(\frac{\alpha+1}{\alpha+2}\right)^{\beta+1}.
\end{equation}
\end{theorem}

\begin{remark}
We observe that Theorem \ref{t:p_concave_r_centroid} is directly obtained from the previous result by just taking
$h=f^{p}$, $\beta=1/p$ and $\alpha=r\beta$ (where the case $r=0$ is derived when doing $\alpha\to0^+$).
Moreover, Theorem \ref{t:functional_centroid_general_alpha_beta} can be shown from Theorem \ref{t:p_concave_r_centroid} by just considering the set of revolution $K$ associated to the radius function $r=(1/\kappa_{n-1})f^{1/(n-1)}$, where $\kappa_{n-1}$ is the $(n-1)$-dimensional volume of the Euclidean unit ball in $\R^{n-1}$, $f=h^\beta$, $p=1/\beta$ and $r=\alpha/\beta$.
Therefore, in fact, both results (Theorems \ref{t:p_concave_r_centroid} and \ref{t:functional_centroid_general_alpha_beta}) are equivalent.
\end{remark}

\smallskip

We would also like to point out that, apart from the already mentioned Theorem \ref{t:p_concave_r_centroid}, and thus in particular Theorem \ref{t:Grunbaum_pgeq0} and Gr\"unbaum's inequality, both Theorems \ref{t:Minkowski_Radon} and \ref{t:Frad_Makai_Martini} can be derived as direct applications of Theorem \ref{t:functional_centroid_general_alpha_beta}.

Indeed, on the one hand, applying Theorem \ref{t:functional_centroid_general_alpha_beta} with $h=f^{1/(n-1)}$, which is concave because of \emph{Brunn's concavity principle} (see e.g. \cite[Section~1.2.1]{BrGiVaVr} and also \cite[Theorem~12.2.1]{Ma}), and taking $\alpha=n-1$ and $\beta\to0^+$, one gets
$b/(b-a)\geq1/(n+1)$, which is exactly \eqref{e:Minkowski_Radon}.

\smallskip

On the other hand, applying Theorem \ref{t:functional_centroid_general_alpha_beta} with $h=f^{1/(n-1)}$ and $\alpha=n-1$, and then raising both sides of \eqref{e:int_ineq_alpha<beta} to the power $1/\beta$ and taking $\beta\to\infty$, one has
\begin{equation}\label{e:deriving_Frad}
\left(\frac{\max_{t\in[0,b]} f(t)}{\|f\|_{\infty}}\right)^{1/(n-1)}\geq \frac{n}{n+1}.
\end{equation}
Since we may assume without loss of generality that $\|f\|_{\infty}>\max_{t\in[0,b]} f(t)$ (considering otherwise the function $t\mapsto f(-t)$), we then get $\max_{t\in[0,b]} f(t)=f(0)$ (since $f$ is $(1/(n-1))$-concave and thus quasi-concave), and therefore \eqref{e:deriving_Frad} is nothing but \eqref{e:Frad_Makai_Martini}.

\smallskip

The paper is organized as follows. In Section \ref{s:background} we collect some preliminaries and background, and we show some first results, such as Propositions \ref{p:midpoint} and \ref{p:first_parcial_result}. The remaining section of the paper, Section \ref{s:proof}, is devoted to the proof of Theorem \ref{t:functional_centroid_general_alpha_beta}, which is also divided into another two subsections, one for each of both cases that are distinguished in Theorem \ref{t:functional_centroid_general_alpha_beta} (namely, $\beta<\alpha$ and $\alpha<\beta$).

\smallskip

\section{Preliminaries and first results}\label{s:background}

We recall that a function $\varphi:\R^n\longrightarrow\R_{\geq0}$ is
$p$-concave, for $p\in\R\cup\{\pm\infty\}$, if
\begin{equation*}\label{e:p-concavecondition}
\varphi\bigl((1-\lambda)x+\lambda y\bigr)\geq
\bigl((1-\lambda)\varphi(x)^p+\lambda \varphi(y)^p\bigr)^{1/p}
\end{equation*}
for all $x,y\in\R^n$ such that $\varphi(x)\varphi(y)>0$ and any $\lambda\in(0,1)$, where the cases $p=0$, $p=\infty$ and $p=-\infty$ must be understood as the corresponding expressions that are obtained by continuity, namely, the geometric mean, the maximum and the minimum (of $\varphi(x)$ and $\varphi(y)$), respectively. 
Note that if $p>0$, then $\varphi$ is $p$-concave if and only if $\varphi^p$ is concave on its support
$\{x\in\R^n: \varphi(x)>0\}$ and thus, in particular, $1$-concave is just concave (on its support) in the usual sense.
A $0$-concave function is usually called
\emph{log-concave} whereas a $(-\infty)$-concave function is referred to as \emph{quasi-concave}.
Moreover, Jensen's inequality for means (see e.g. \cite[Section~2.9]{HaLiPo} and \cite[Theorem~1 p.~203]{Bu}) implies that a $q$-concave function is also $p$-concave, whenever $q>p$.

\smallskip

For the sake of simplicity, in the following we consider
$H=\{x\in\R^n: \esc{x,u}=0\}$, for a given direction $u\in\s^{n-1}$ that
we extend to an orthonormal basis $(u_1,u_2,\dots,u_n)$ of $\R^n$, with $u_1=u$.
Given a set $M\subset\R^n$, and a real number $c\in\R$, we will write $M^-(u,c)=M\cap\{x\in\R^n: \esc{x,u}\leq c\}$ and
$M^+(u,c)=M\cap\{x\in\R^n: \esc{x,u}\geq c\}$. When 
dealing with the above-mentioned prescribed hyperplane $H$ we will just write $H^-$ and $H^+$ for the corresponding halfspaces determined by $H$.
Moreover, for a compact set with non-empty interior $K\subset\R^n$, we denote by $K(t)=K\cap(tu+H)$ for any $t\in\R$. We notice that, if $K|H^\bot\subset[au, bu]$, Fubini's theorem
applied to the function $f:\R\longrightarrow\R_{\geq0}$ given by \[f(t)=\vol_{n-1}\bigl(K\cap(tu+H)\bigr)\]
yields (provided that $a\leq0$) 
\begin{equation}\label{e:vol(K)_Fub}
\vol(K)=\int_a^b f(t)\,\dlat t
\quad \text{and} \quad \vol(K^-)=\int_a^0 f(t)\,\dlat t,
\end{equation}
where, as usual, we are identifying the linear subspace spanned by $u$ with $\R$.
Since the set $\{t\in\R: f(t)>0\}$ is convex whenever $f$ is quasi-concave, from now on we will assume, without loss of generality, that $f(t)>0$ for all $t\in(a,b)$. Furthermore, by Fubini's theorem, we get
\begin{equation}\label{e:g(K)_first_comp}
[\g(K)]_1=\frac{1}{\vol(K)}\int_a^b tf(t)\,\dlat t
\end{equation}
and thus, in particular, $a<[\g(K)]_1<b$ (cf. \eqref{e:vol(K)_Fub}).

It seems natural to wonder about the possibility of getting an analogue of Theorem \ref{t:Grunbaum_pgeq0} by considering the midpoint in the direction $u$ of $K$ instead its centroid, that is, the point $\bigl[(a+b)/2\bigr]\cdot u$.
This is the content of the following result.

\begin{proposition}\label{p:midpoint}
Let $K\subset\R^n$ be a compact set with non-empty interior and with midpoint, with respect to some direction $u\in\s^{n-1}$, at the origin. Let $H=\{x \in \R^{n}: \esc{x,u}=0\}$ be the hyperplane with normal vector $u$ and assume that the function $f:H^{\bot}\longrightarrow\R_{\geq0}$ given by $f(x)=\vol_{n-1}\bigl(K\cap(x+H)\bigr)$ is $p$-concave, for some $p\in(0,\infty)$. Then
\begin{equation*}
\frac{\vol(K^{-})}{\vol(K)}\geq\left(\frac{1}{2}\right)^{(p+1)/p}.
\end{equation*}
\end{proposition}

\begin{proof}
In the proof of \cite[Theorem~1.1]{MSYN} it is shown that
there exists a $p$-affine function $g_p:[-\gamma,\delta]\longrightarrow\R_{\geq0}$ given by
$g_p(t)=c(t+\gamma)^{1/p}$, for some $\gamma,\delta,c>0$,
such that $g_p(0)=f(0)$,
\begin{equation*}\label{e:eq_integ_f_g_p>0}
\int_{-\gamma}^0 g_p(t)\,\dlat t=\int_a^0 f(t)\,\dlat t \quad \text{and} \quad
\int_0^\delta g_p(t)\,\dlat t=\int_0^b f(t)\,\dlat t,
\end{equation*}
and further that $-\gamma\leq a<0<\delta\leq b$.

Here, since $K$ has its midpoint (w.r.t. $u$) at the origin, we have that $a=-b$ and then we get $-\gamma+\delta\leq0$. Thus
\[\frac{\vol(K^-)}{\vol(K)}=\frac{\int_{-\gamma}^0 g_p(t)\,\dlat t}{\int_{-\gamma}^\delta g_p(t)\,\dlat t}
\geq\frac{\int_{-\gamma}^{(-\gamma+\delta)/2} g_p(t)\,\dlat t}{\int_{-\gamma}^\delta g_p(t)\,\dlat t}
=\left(\frac{1}{2}\right)^{(p+1)/p},
\]
as desired.
\end{proof}

We collect here the following result, originally proved in
\cite{Borell} and \cite{BL} (see also \cite{G} for a detailed presentation), which can be regarded as the
functional counterpart of the \emph{Brunn-Minkowski inequality}.
\begin{thm}[The Borell-Brascamp-Lieb inequality]\label{t:BBL}
Let $\lambda\in(0,1)$. Let $-1/n\leq p\leq\infty$ and let $f,g,h:\R^n\longrightarrow\R_{\geq0}$ be measurable
functions with positive integrals such that
\begin{equation*}
h((1-\lambda)x + \lambda y)\geq \bigl((1-\lambda)f(x)^p+\lambda g(y)^p\bigr)^{1/p}
\end{equation*}
for all $x,y\in\R^n$ with $f(x)g(y)>0$. Then
\begin{equation}\label{e:BBL}
\int_{\R^n}h(x)\,\dlat x\geq
\left[(1-\lambda)\left(\int_{\R^n}f(x)\,\dlat x\right)^q+\lambda\left(\int_{\R^n}g(x)\,\dlat x\right)^q\right]^{1/q},
\end{equation}
where $q=p/(np+1)$.
\end{thm}

\smallskip

Now, considering the points $\centroid{r}\cdot u$, for any $r>0$, where $\centroid{r}$ is given by \eqref{e:lambda_r},
and following the same idea as in \cite[Theorem~3]{Fr}, we can get a first result concerning this family of points. The statement reads as follows.

\begin{corollary}\label{c:first_parcial_result}
Let $r\in(0,\infty)$ and let $K\subset\R^n$ be a compact set with non-empty interior having the point $\centroid{r}\cdot u$, with respect to some direction $u\in\s^{n-1}$, at the origin. Let $H=\{x \in \R^{n}: \esc{x,u}=0\}$ be the hyperplane with normal vector $u$ and assume that the function $f:\R\longrightarrow\R_{\geq0}$ given by $f(t)=\vol_{n-1}\bigl(K\cap(tu+H)\bigr)$ is $p$-concave, for some $p\in(0,\infty)$. Then
\begin{equation}\label{e:first_parcial_result}
\frac{\vol(K^{-})}{\vol(K)} \geq \left(\frac{p}{2p + r}\right)^{(p+1)/p}.
\end{equation}
\end{corollary}
We will derive Corollary \ref{c:first_parcial_result} as a simple application of the following (slightly more general) functional result.
\begin{proposition}\label{p:first_parcial_result}
Let $K \subset \R^n$ be a convex body. Let $g : K \longrightarrow \R_{\geq 0}$ be a concave function and let $f : K \longrightarrow \R_{\geq 0}$ be a $p$-concave function, with $p>0$. Then
\begin{equation*}
g\left(\frac{\int_{K} x f(x) \,\dlat x}{\int_{K} f(x) \,\dlat x}\right) \geq \left(\frac{p}{(n + 1)p + 1}\right)\|g\|_{\infty}.
\end{equation*}
\end{proposition}
\begin{proof}
Let $\mu$ be the probability measure whose density function is given by
\[
\dlat\mu(x) = \frac{f(x)}{\int_K f(x) \,\dlat x}\dlat x.
\]
Since $g$ is concave, using Jensen's inequality we get that
\begin{equation*}
\begin{split}
    g\left(\frac{\int_{K} x f(x) \,\dlat x}{\int_{K} f(x) \,\dlat x}\right) = g\left(\int_K x \, \dlat\mu(x)\right) 
    &\geq \int_K g(x) \,\dlat\mu(x)\\ 
    &= \int_0^{\|g\|_{\infty}}\mu\bigl(\{x \in K : g(x) \geq t\}\bigr)\,\dlat t,
\end{split}
\end{equation*}
where in the last identity we have used Fubini's theorem. Now, since the density of $\mu$, with respect to the Lebesgue measure, is $p$-concave,  from the Borell-Brascamp-Lieb inequality \eqref{e:BBL} we have that
the function $\varphi(t) := \mu\bigl(\{x \in K : g(x) \geq t\}\bigr)$ is $\bigl(p/(np + 1)\bigr)$-concave. 
Indeed, it is enough to apply Theorem \ref{t:BBL} with the functions 
\[
\frac{f\cdot\chi_{_{\{x\in K\,:\, g(x)\geq t_1\}}}}{\int_K f(x) \,\dlat x}, \quad \frac{f\cdot\chi_{_{\{x\in K\,:\, g(x)\geq t_2\}}}}{\int_K f(x) \,\dlat x} \quad \text{ and } \quad \frac{f\cdot\chi_{_{\{x\in K\,:\, g(x)\geq (1-\lambda)t_1+\lambda t_2\}}}}{\int_K f(x) \,\dlat x}
\] 
for any $t_1,t_2\in\R$, where $\chi_{_M}$ denotes the characteristic function of the set $M$, one easily obtains the desired concavity of $\varphi$.
Hence, and taking into account that $\varphi(0)=1$ and $\varphi\bigl(\|g\|_{\infty}\bigr)\geq0$, we may assure that $\varphi(t)^{p/(np + 1)} \geq (1 - t/\|g\|_{\infty})$ for all $t \in [0,\|g\|_{\infty}]$. So, by integrating we obtain that
\begin{equation*}
\begin{split}
\int_0^{\|g\|_{\infty}}\mu\bigl(\{x \in K : g(x) \geq t\}\bigr)\,\dlat t 
&\geq \int_0^{\|g\|_{\infty}}(1 - t/\|g\|_{\infty})^{(np + 1)/p}\,\dlat t\\
&= \left(\frac{p}{(n + 1)p + 1}\right)\|g\|_{\infty},
\end{split}
\end{equation*}
from where the result immediately follows.
\end{proof}

\smallskip

We conclude this section by showing Corollary \ref{c:first_parcial_result}.

\begin{proof}[Proof of Corollary \ref{c:first_parcial_result}]
Denoting by $[a,b]$ the support of $f$, let $\bar{f},g:[a,b]\longrightarrow\R_{\geq0}$ be the functions given by
\[
g(t)=\vol\bigl(K \cap (tu + H^{-})\bigr)^{p/(p+1)}=\left(\int_a^t f(s)\,\dlat s\right)^{p/(p+1)}
\] 
and 
\[
\bar{f}(t)=\vol_{n-1}\bigl(K\cap(tu+H)\bigr)^r=f(t)^r.
\]
By hypothesis, it is clear that $\bar{f}$ is $(p/r)$-concave, whereas from the Borell-Brascamp-Lieb inequality \eqref{e:BBL} we have that $g$ is concave on $[a,b]$. So, from Proposition \ref{p:first_parcial_result} applied to the functions $\bar{f}$ and $g$, and taking into account that $\centroid{r}=0$, we get that
\begin{equation*}
\begin{split}
\vol(K^{-}) = g\left(\frac{\int_a^b t \bar{f}(t) \,\dlat t}{\int_a^b \bar{f}(t)\,\dlat t}\right)^{(p+1)/p} 
&\geq \left(\frac{p}{2p + r}\right)^{(p+1)/p}\|g\|_{\infty}^{(p+1)/p}\\
&= \left(\frac{p}{2p + r}\right)^{(p+1)/p}\vol(K),
\end{split}
\end{equation*}
as desired.
\end{proof}

In the following section, the tighter inequalities collected in Theorem \ref{t:p_concave_r_centroid}, which improve the one that was obtained in Corollary \ref{c:first_parcial_result}, will be shown. 
More precisely, we will enhance the constant given by \eqref{e:first_parcial_result}, namely, $\bigl(p/(2p + r)\bigr)^{(p+1)/p}$, by other two constants (depending on whether $r$ is either less than one or greater than one). Furthermore, these new constants do fit well with \eqref{e:Grunbaum_p>0}, since, in fact, they all coincide when $r=1$, that is, when one considers the centroid of the set.

\section{Proof of Theorem \ref{t:functional_centroid_general_alpha_beta}}\label{s:proof}
Here we are going to slightly modify the approach followed in \cite[Theorem~8]{StZh} (which yields our case $\beta=n-1$) to cover all the cases for a general concave function $h$, which will allow us to show Theorem \ref{t:functional_centroid_general_alpha_beta}. To this aim, we split the proof into two steps, depending on whether $\beta<\alpha$ or $\alpha<\beta$. Note also that the case $\alpha=\beta$ is equivalent to the statement of Theorem \ref{t:Grunbaum_pgeq0} (by just taking $p:=1/\alpha=1/\beta$ and $f=h^\alpha=h^\beta$).
Before distinguishing whether $\beta<\alpha$ or $\alpha<\beta$, we make some general considerations.

We may assume, without loss of generality, that $a=0$. Now, let $L\subset\R^2$ be the convex body
\[
L:=\bigl\{(x,y)\in\R^2\,:\,0\leq x\leq b,\,0\leq y\leq h(x)\bigr\}
\]
and notice that, from Fubini's theorem, we have
\begin{equation}\label{e:centroide_alpha(h)}
\centroid{\alpha}(h)=\frac{\int_0^b th(t)^\alpha\,\dlat t}{\int_0^b h(t)^\alpha\,\dlat t}=\frac{\int_L\langle x,\e_1\rangle\langle x,\e_2\rangle^{\alpha-1}\,\dlat x}{\int_L\langle x,\e_2\rangle^{\alpha-1}\,\dlat x}.
\end{equation}
Let $\mu_\beta$ be the measure on $\R^2$ given by $\,\dlat\mu_\beta(x)=\langle x,\e_2\rangle^{\beta-1} \,\dlat x$. Then
\[
\centroid{\alpha}(h)=\frac{\int_L\langle x,\e_1\rangle\langle x,\e_2\rangle^{\alpha-\beta}\,\dlat\mu_\beta(x)}{\int_L\langle x,\e_2\rangle^{\alpha-\beta}\,\dlat\mu_\beta(x)}
\]
and
\begin{equation*}
\begin{split}
\frac{\int_{\centroid{\alpha}(h)}^b h(t)^\beta\,\dlat t}{\int_{0}^b h(t)^\beta\,\dlat t}
&= \frac{\int_{\{x\in L\,:\,\langle x,\e_1\rangle\geq \centroid{\alpha}(h)\}}\langle x,\e_2\rangle^{\beta-1}\,\dlat x} {\int_L\langle x,\e_2\rangle^{\beta-1}\,\dlat x}\\
&=\frac{\mu_\beta\{x\in L\,:\,\langle x,\e_1\rangle\geq \centroid{\alpha}(h)\}}{\mu_\beta(L)}.
\end{split}
\end{equation*}
We may assure that there exist $\gamma<\delta$ and $c>0$ in such a way that the affine decreasing function $g:[\gamma,\delta]\longrightarrow[0,\infty)$ given by
\[
g(t)=c(\delta-t)
\]
satisfies
\begin{equation}\label{e:conditions_decreasing_affine}
\begin{array}{l}
\text{(i)}\; g\bigl(\centroid{\alpha}(h)\bigr)=h\bigl(\centroid{\alpha}(h)\bigr),\\[4mm]
\text{(ii)}\; \int_{\gamma}^\delta g(t)^\beta\,\dlat t=\int_{0}^b h(t)^\beta\,\dlat t, \quad\text{ and }\\[4mm]
\text{(iii)}\; \int_{\centroid{\alpha}(h)}^\delta g(t)^\beta\, \dlat t =
\int_{\centroid{\alpha}(h)}^b h(t)^\beta\,\dlat t.
\end{array}
\end{equation}
Indeed, taking
\[
\delta=\frac{\beta+1}{h\bigl(\centroid{\alpha}(h)\bigr)^\beta}\int_{\centroid{\alpha}(h)}^b h(t)^{\beta}\,\dlat t + \, \centroid{\alpha}(h),
\quad c=\frac{h\bigl(\centroid{\alpha}(h)\bigr)}{\delta - \centroid{\alpha}(h)}
\]
and 
\[\gamma= \delta - \left(\frac{\beta+1}{c^\beta}\int_0^b h(t)^\beta\,\dlat t\right)^{1/(\beta+1)},\]
elementary computations show \eqref{e:conditions_decreasing_affine}.

\smallskip

\begin{figure}[h]
\centering
\includegraphics[width=\textwidth]{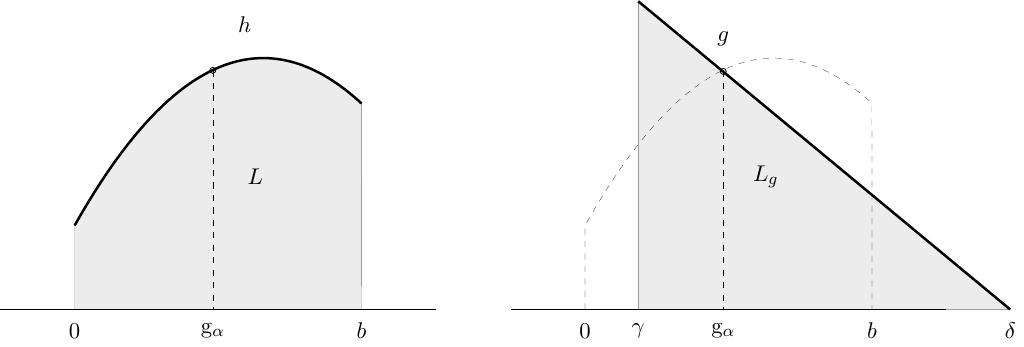}
\caption{Sets $L$ and $L_g$.}
\label{f:L_and_L_g}
\end{figure}

\medskip

Then, denoting by $L_g\subset\R^2$ the triangle (see Figure \ref{f:L_and_L_g}) given by
\[
L_g:=\bigl\{(x,y)\in\R^2\,:\,\gamma\leq x\leq \delta,\,0\leq y\leq g(x)\bigr\},
\]
from (ii) and (iii) in \eqref{e:conditions_decreasing_affine}  (by using Fubini's theorem), and the relative position of $h$ and $g$, given by the concavity of $h$ and the relation $g\bigl(\centroid{\alpha}(h)\bigr)=h\bigl(\centroid{\alpha}(h)\bigr)$ (see Figure \ref{f:h_and_g}), we have that
\begin{equation}\label{e:conditions_decreasing_affine_II}
\begin{array}{l}
\text{(i)}\; \mu_\beta (L)=\mu_\beta(L_g),\\[3mm]
\text{(ii)}\; \mu_\beta \bigl(\{x\in L:\langle x,\e_1\rangle\geq \centroid{\alpha}(h)\}\bigr)
\!=\!\mu_\beta\bigl(\{x\in L_g:\langle x,\e_1\rangle\geq \centroid{\alpha}(h)\}\bigr),\\[3mm]
\text{(iii)}\; 0\leq\gamma\leq \centroid{\alpha}(h)\leq b\leq\delta.
\end{array}
\end{equation}

\smallskip

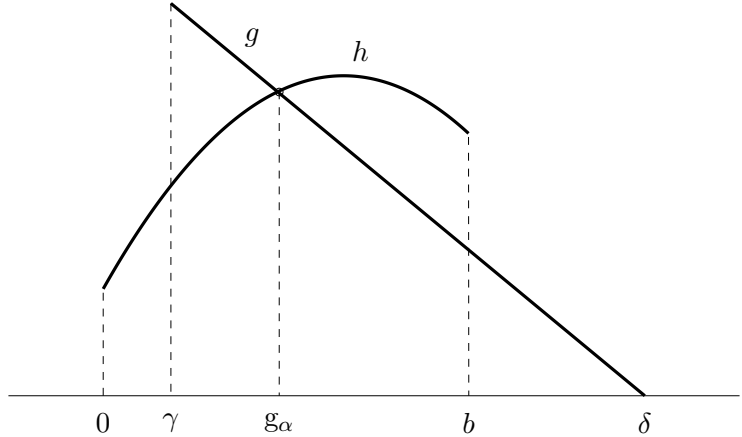
\begin{figure}[h]
\centering
\begin{tikzpicture}[scale=1.8]
\draw[ultra thin,dashed] (-0.2,2.9) -- (-0.2,0);
\draw[very thick] (-0.2,2.9) -- (3.3,0);
\draw[ultra thin, dashed, domain=-0.7:2, samples = 1000, variable=\x] (-0.7,0) -- plot({\x}, {-0.5*(\x+1.1)*(\x-3.25)}) -- (2,0);
\draw[very thick, domain=-0.7:2, samples = 1000, variable=\x] plot({\x}, {-0.5*(\x+1.1)*(\x-3.25)});
\draw[very thin] (-1.4,0) -- (4,0);
\draw (0.4,2.65) node{$g$};
\draw (1.2,2.55) node{$h$};
\draw (-0.2,-0.2) node{$\gamma$};
\draw (3.3,-0.2) node{$\delta$};
\draw (-0.7,-0.2) node{$0$};
\draw (2,-0.2) node{$b$};
\draw (0.6,2.245) circle (0.8pt);
\draw[ultra thin, dashed] (0.6,2.25) -- (0.6,0);
\draw (0.6,-0.2) node{$\mathrm{g}_{\alpha}$};
\end{tikzpicture}
\caption{Relative position of the functions $h$ and $g$.}
\label{f:h_and_g}
\end{figure}

\medskip

Moreover, defining $g(t):=0$ for all $t\in[0,\gamma]$ and $h(t):=0$ for any $t\in[b,\delta]$, there exists $x_0\in(\centroid{\alpha}(h), b]$ such that $h(t)\geq g(t)$ for all $t\in [0,\gamma] \cup [\centroid{\alpha}(h), x_0]$  and $h(t)\leq g(t)$ otherwise (see Figure \ref{f:h_and_g} -there, observe that $x_0$ coincides with $b$). Hence, on the one hand, for every $s \in [\centroid{\alpha}(h), x_0]$
(the case of $s \geq x_0$ immediately follows) we get that
\begin{equation*}
\begin{split}
\int_{s}^{b} h(t)^{\beta}\,\dlat t &= \int_{\centroid{\alpha}(h)}^{b} h(t)^{\beta}\,\dlat t - \int_{\centroid{\alpha}(h)}^{s}  h(t)^{\beta}\,\dlat t\\
&\leq \int_{\centroid{\alpha}(h)}^{\delta} g(t)^{\beta}\,\dlat t - \int_{\centroid{\alpha}(h)}^{s} g(t)^{\beta}\,\dlat t = \int_{s}^{\delta} g(t)^{\beta}\,\dlat t.
\end{split}
\end{equation*}
On the other hand, for every $s \in [\gamma, \centroid{\alpha}(h)]$ (again, the case of $s \leq \gamma$ immediately follows) we have that
\begin{equation*}
\begin{split}
\int_{s}^{b} h(t)^{\beta}\,\dlat t &= \int_{\centroid{\alpha}(h)}^{b} h(t)^{\beta}\,\dlat t + \int_s^{\centroid{\alpha}(h)} h(t)^{\beta}\,\dlat t\\
&\leq \int_{\centroid{\alpha}(h)}^{\delta} g(t)^{\beta}\,\dlat t + \int_s^{\centroid{\alpha}(h)} g(t)^{\beta}\,\dlat t = \int_{s}^{\delta} g(t)^{\beta}\,\dlat t.
\end{split}
\end{equation*}
Therefore,
\begin{equation}\label{e:ineq_beta_measures}
\begin{split}
\mu_\beta\bigl(L^+(\e_1,s)\bigr)&=\mu_\beta\bigl(\{x\in L\,:\,\langle x,\e_1\rangle\geq s\}\bigr) = \int_{s}^{b} h(t)^{\beta}\,\dlat t\\
&\leq \int_{s}^{\delta} g(t)^{\beta}\,\dlat t = \mu_\beta\bigl(\{x\in L_g\,:\,\langle x,\e_1\rangle\geq s\}\bigr)=\mu_\beta\bigl(L_g^+(\e_1,s)\bigr)
\end{split}
\end{equation}
for every $s\in[0,\delta]$, where we are using the notation $L_g^+(u,c)$ to represent the set $\bigl(L_g\bigr)^+(u,c)$, for a given direction $u\in\s^{n-1}$ and a real number $c\in\R$.

\subsection{The case of $\beta<\alpha$}\label{s:proof_beta<alpha}
We devote this section to proving the first part of Theorem \ref{t:functional_centroid_general_alpha_beta}, namely, we show \eqref{e:int_ineq_beta<alpha} provided that 
$\beta<\alpha$.

\smallskip

We will first prove that there exists a non-negative and concave function $\varphi:[\gamma,\delta]\longrightarrow[0,\|h\|_\infty]$ such that
\begin{equation}\label{e:bounding_centroide_alpha(h)}
\centroid{\alpha}(h)\leq\frac{\int_{L_g}\langle x,\e_1\rangle\varphi(\langle x,\e_1\rangle)^{\alpha-\beta} \,\dlat\mu_\beta(x)}{\int_{L_g}\varphi(\langle x,\e_1\rangle)^{\alpha-\beta}\, \dlat\mu_\beta(x)}.
\end{equation}
To this aim, we consider the function $W:[0,\|h\|_\infty]\longrightarrow[0,\mu_\beta(L)]$ given by
\[
W(s)=\mu_\beta\bigl(\{x\in L\,:\,\langle x,\e_2\rangle\geq s\}\bigr),
\]
which is clearly both strictly decreasing and surjective.
We may then define the non-negative function $w:[0,\|h\|_\infty]\longrightarrow[\gamma,\delta]$ that satisfies
\[
W(s)=\mu_\beta\bigl(\{x\in L_g\,:\,\langle x,\e_1\rangle\geq w(s)\}\bigr)
\]
for any $s\in[0,\|h\|_\infty]$.
Indeed, since
\[
W(s)=\frac{c^\beta(\delta-w(s))^{\beta+1}}{\beta(\beta+1)},
\]
we get that
\[
w(s)=\delta-\left(\frac{\beta(\beta+1)}{c^\beta}W(s)\right)^{1/(\beta+1)}.
\]
Notice that, from the Borell-Brascamp-Lieb inequality \eqref{e:BBL}, we have that
the function $W^{1/(\beta+1)}$ is concave (since the density of $\mu_\beta$, with respect to the Lebesgue measure, is $\bigl(1/(\beta-1)\bigr)$-concave).
Therefore, $w$ is strictly increasing, surjective and convex, and then there exists the function $\varphi=w^{-1}:[\gamma,\delta]\longrightarrow[0,\|h\|_\infty]$, which is further (strictly increasing and) concave.

\smallskip

Now, we will start by bounding from above the right-hand side of \eqref{e:centroide_alpha(h)}.
By using Fubini's theorem, (iii) in \eqref{e:conditions_decreasing_affine_II} and \eqref{e:ineq_beta_measures}, we have
\begin{equation}\label{e:bounding_numer_centroide_alpha(h)}
\begin{split}
\frac{1}{\alpha-\beta}&\int_L\langle x,\e_1\rangle\langle x,\e_2\rangle^{\alpha-\beta}\,\dlat\mu_\beta(x)\\
=&\int_L\int_0^{\langle x,\e_1\rangle}\,\dlat s_1\int_0^{\langle x,\e_2\rangle}s_2^{\alpha-\beta-1}
\,\dlat s_2\,\dlat\mu_\beta(x)\\
=&\int_0^{b}\int_0^{\|h\|_\infty}s_2^{\alpha-\beta-1}\mu_\beta\bigl(L^+(\e_1,s_1)\cap L^+(\e_2,s_2)\bigr)\,\dlat s_2\,\dlat s_1\\
\leq&\int_0^{b}\int_0^{\|h\|_\infty}s_2^{\alpha-\beta-1}\min\left\{\mu_\beta\bigl(L^+(\e_1,s_1)\bigr),
\mu_\beta\bigl(L^+(\e_2,s_2)\bigr)\right\}\,\dlat s_2\,\dlat s_1\\
\leq&\int_0^{b}\int_0^{\|h\|_\infty}s_2^{\alpha-\beta-1}\min\left\{\mu_\beta\bigl(L_g^+(\e_1,s_1)\bigr),
\mu_\beta\bigl(L^+(\e_2,s_2)\bigr)\right\}\,\dlat s_2\,\dlat s_1\\
\leq&\int_0^{\delta}\int_0^{\|h\|_\infty}s_2^{\alpha-\beta-1}\min\left\{\mu_\beta\bigl(L_g^+(\e_1,s_1)\bigr),
\mu_\beta\bigl(L^+(\e_2,s_2)\bigr)\right\}\,\dlat s_2\,\dlat s_1.
\end{split}
\end{equation}

So, on the one hand, from \eqref{e:bounding_numer_centroide_alpha(h)} we get
\begin{equation}\label{e:upper_bound_numer_centroide_alpha(h)}
\begin{split}
\frac{1}{\alpha-\beta}&\int_L\langle x,\e_1\rangle\langle x,\e_2\rangle^{\alpha-\beta}\,\dlat\mu_\beta(x)\\
\leq&\int_0^{\delta}\int_0^{\|h\|_\infty}s_2^{\alpha-\beta-1}\min\left\{\mu_\beta\bigl(L_g^+(\e_1,s_1)\bigr),
\mu_\beta\bigl(L^+(\e_2,s_2)\bigr)\right\}\,\dlat s_2\,\dlat s_1\\
=&\int_0^{\delta}\int_0^{\|h\|_\infty}s_2^{\alpha-\beta-1}\min\left\{\mu_\beta\bigl(L_g^+(\e_1,s_1)\bigr),
\mu_\beta\Bigl(L_g^+\bigl(\e_1,w(s_2)\bigr)\Bigr)\right\}\,\dlat s_2\,\dlat s_1\\
=&\int_0^{\delta}\int_0^{\|h\|_\infty}s_2^{\alpha-\beta-1}\mu_\beta\Bigl(L_g^+(\e_1,s_1)\cap
L_g^+\bigl(\e_1,w(s_2)\bigr)\Bigr)\,\dlat s_2\,\dlat s_1\\
=&\,\frac{1}{\alpha-\beta}\int_{L_g}\langle x,\e_1\rangle\varphi(\langle x,\e_1\rangle)^{\alpha-\beta}\,\dlat\mu_\beta(x),
\end{split}
\end{equation}
where in the last equality above we have used that $\langle x,\e_1\rangle\geq w(s_2)$ if and only if $\varphi(\langle x,\e_1\rangle)\geq s_2$.

\medskip

On the other hand, since $\alpha-\beta>0$, from Fubini's theorem we have that
\begin{equation}\label{e:denom_centroide_alpha(h)}
\begin{split}
\frac{1}{\alpha-\beta}\int_L\langle x,\e_2\rangle^{\alpha-\beta}\,\dlat\mu_\beta(x)
=&\int_0^{\|h\|_\infty}s^{\alpha-\beta-1}\mu_\beta\bigl(\{x\in L\,:\,\langle x,\e_2\rangle\geq s\}\bigr)\,\dlat s\\[1mm]
=&\int_0^{\|h\|_\infty}s^{\alpha-\beta-1}\mu_\beta\bigl(\{x\in L_g\,:\,\langle x,\e_1\rangle\geq w(s)\}\bigr)\,\dlat s\\[1mm]
=&\int_0^{\|h\|_\infty}s^{\alpha-\beta-1}\mu_\beta\bigl(\{x\in L_g\,:\,\varphi(\langle x,\e_1\rangle)\geq s\}\bigr)
\,\dlat s\\[1mm]
=&\int_{L_g}\varphi(\langle x,\e_1\rangle)^{\alpha-\beta}\,\dlat\mu_\beta(x).
\end{split}
\end{equation}
Hence, from \eqref{e:upper_bound_numer_centroide_alpha(h)} and \eqref{e:denom_centroide_alpha(h)} (and using \eqref{e:centroide_alpha(h)}) we obtain \eqref{e:bounding_centroide_alpha(h)}, as desired.

\medskip

Now we will prove that for any concave function $\varphi:[\gamma,\delta]\longrightarrow[0,\infty)$ we have that
\begin{equation}\label{e:ineq_L_g_varphi}
\frac{\int_{L_g}\langle x,\e_1\rangle\varphi(\langle x,\e_1\rangle)^{\alpha-\beta} \,\dlat\mu_\beta(x)}{\int_{L_g}\varphi(\langle x,\e_1\rangle)^{\alpha-\beta}\,\dlat\mu_\beta(x)}
\leq \frac{\int_{L_g}\langle x,\e_1\rangle\left(\langle x,\e_1\rangle-\gamma\right)^{\alpha-\beta} \,\dlat\mu_\beta(x)}{\int_{L_g}\left(\langle x,\e_1\rangle-\gamma\right)^{\alpha-\beta}\, \dlat\mu_\beta(x)}.
\end{equation}
To this aim, let $C_1>0$ be such that
\begin{equation}\label{e:C_1_L_g_varphi}
\int_{L_g}\bigl(C_1(\langle x,\e_1\rangle-\gamma)\bigr)^{\alpha-\beta}\,\dlat\mu_\beta(x)
=\int_{L_g}\varphi(\langle x,\e_1\rangle)^{\alpha-\beta}\,\dlat\mu_\beta(x).
\end{equation}
Since the latter identity is equivalent (by Fubini's theorem) to
\begin{equation*}
\int_\gamma^{\delta}\left(\bigl(C_1(t-\gamma)\bigr)^{\alpha-\beta}
-\varphi(t)^{\alpha-\beta}\right)
g(t)^\beta\,\dlat t=0,
\end{equation*}
we may assert, taking into account that $\varphi$ is concave, that
there exists $t_0\in(\gamma,\delta)$ such that
\begin{equation}\label{e:C_1_L_g_varphi2}
\begin{array}{l}
\text{(i)}\; C_1(t-\gamma)\leq\varphi(t) \text{ for every } \gamma\leq t\leq t_0, \quad \text{and}\\[6mm]
\text{(ii)}\; C_1(t-\gamma)\geq\varphi(t) \text{ for every } t_0\leq t\leq \delta.
\end{array}
\end{equation}
Then, from \eqref{e:C_1_L_g_varphi} and \eqref{e:C_1_L_g_varphi2} (and using Fubini's theorem), we get
\begin{equation*}
\begin{split}
\beta\biggl(\int_{L_g}\langle x,\e_1\rangle&\varphi(\langle x,\e_1\rangle)^{\alpha-\beta}\,\dlat\mu_\beta(x)
-\int_{L_g}\langle x,\e_1\rangle\bigl(C_1(\langle x,\e_1\rangle-\gamma)\bigr)^{\alpha-\beta}\,\dlat\mu_\beta(x)\biggr)\\
=&\int_\gamma^\delta t\Bigl(\varphi(t)^{\alpha-\beta}-\bigl(C_1(t-\gamma)\bigr)^{\alpha-\beta}\Bigr)g(t)^\beta\,\dlat t\\
=&\int_\gamma^{t_0} t\Bigl(\varphi(t)^{\alpha-\beta}-\bigl(C_1(t-\gamma)\bigr)^{\alpha-\beta}\Bigr)g(t)^\beta\,\dlat t\\
+&\int_{t_0}^\delta t\Bigl(\varphi(t)^{\alpha-\beta}-\bigl(C_1(t-\gamma)\bigr)^{\alpha-\beta}\Bigr)g(t)^\beta\,\dlat t\\
\leq& \,t_0\int_\gamma^{t_0}\left(\varphi(t)^{\alpha-\beta}-\bigl(C_1(t-\gamma)\bigr)^{\alpha-\beta}\right)
g(t)^\beta\,\dlat t\\
+& \,t_0\int_{t_0}^\delta \left(\varphi(t)^{\alpha-\beta}-\bigl(C_1(t-\gamma)\bigr)^{\alpha-\beta}\right)g(t)^\beta\,\dlat t\\
=& \,t_0\int_\gamma^{\delta}\left(\varphi(t)^{\alpha-\beta}-\bigl(C_1(t-\gamma)\bigr)^{\alpha-\beta}\right)
g(t)^\beta\,\dlat t\\
=& \,\beta t_0\int_{L_g}\left(\varphi(\langle x,\e_1\rangle)^{\alpha-\beta}-
\bigl(C_1(\langle x,\e_1\rangle-\gamma)\bigr)^{\alpha-\beta}\right)\,\dlat\mu_\beta(x)=0.
\end{split}
\end{equation*}
Thus, we have
\begin{equation*}
\int_{L_g}\langle x,\e_1\rangle\varphi(\langle x,\e_1\rangle)^{\alpha-\beta}\,\dlat\mu_\beta(x)
\leq \int_{L_g}\langle x,\e_1\rangle\bigl(C_1(\langle x,\e_1\rangle-\gamma)\bigr)^{\alpha-\beta} \,\dlat\mu_\beta(x),
\end{equation*}
which, together with \eqref{e:C_1_L_g_varphi}, yields \eqref{e:ineq_L_g_varphi}.

\smallskip

Now, we will compute the right-hand side of \eqref{e:ineq_L_g_varphi}. On the one hand,
\begin{equation*}
\begin{split}
\int_{L_g}\left(\langle x,\e_1\rangle-\gamma\right)^{\alpha-\beta}\,\dlat\mu_\beta(x) =&\frac{c^\beta}{\beta}\int_\gamma^\delta(t-\gamma)^{\alpha-\beta}(\delta-t)^\beta\,\dlat t\\ =&\frac{c^\beta}{\beta}(\delta-\gamma)^{\alpha+1}\int_0^1s^{\alpha-\beta}(1-s)^\beta\,\dlat s\\
=&\frac{c^\beta}{\beta}(\delta-\gamma)^{\alpha+1}
\frac{\Gamma\left(\alpha-\beta+1\right)\Gamma\left(\beta+1\right)}{\Gamma\left(\alpha+2\right)},
\end{split}
\end{equation*}
whereas, on the other hand, we obtain
\begin{equation*}
\begin{split}
&\quad\quad\int_{L_g}\langle x,\e_1\rangle\left(\langle x,\e_1\rangle-\gamma\right)^{\alpha-\beta}\,\dlat\mu_\beta(x)
=\frac{c^\beta}{\beta}\int_\gamma^\delta t(t-\gamma)^{\alpha-\beta}(\delta-t)^\beta\,\dlat t\\
&=\gamma\frac{c^\beta}{\beta}(\delta-\gamma)^{\alpha+1}\int_0^1s^{\alpha-\beta}(1-s)^\beta\,\dlat s+\frac{c^\beta}{\beta}(\delta-\gamma)^{\alpha+2}\int_0^1s^{\alpha-\beta+1}(1-s)^\beta\,\dlat s\\
&=\frac{c^\beta}{\beta}(\delta-\gamma)^{\alpha+1}
\frac{\Gamma\left(\alpha-\beta+1\right)\Gamma\left(\beta+1\right)}{\Gamma\left(\alpha+2\right)}
\left(\gamma+(\delta-\gamma)\frac{\alpha-\beta+1}{\alpha+2}\right).
\end{split}
\end{equation*}
Hence, we have
\[
\frac{\int_{L_g}\langle x,\e_1\rangle\left(\langle x,\e_1\rangle-\gamma\right)^{\alpha-\beta} \,\dlat\mu_\beta(x)}
{\int_{L_g}\left(\langle x,\e_1\rangle-\gamma\right)^{\alpha-\beta}\,\dlat\mu_\beta(x)}
=\gamma+(\delta-\gamma)\frac{\alpha-\beta+1}{\alpha+2},
\]
and therefore, this together with \eqref{e:bounding_centroide_alpha(h)} and \eqref{e:ineq_L_g_varphi}
yields
\[
\centroid{\alpha}(h)\leq \gamma+(\delta-\gamma)\frac{\alpha-\beta+1}{\alpha+2}=:g_0.
\]
Finally, the latter relation jointly with (ii) and (iii) in \eqref{e:conditions_decreasing_affine} gives us
\begin{equation*}
\begin{split}
\frac{\int_{\centroid{\alpha}(h)}^b h(t)^\beta\,\dlat t}{\int_{a}^bh(t)^\beta\,\dlat t}
&=\frac{\int_{\centroid{\alpha}(h)}^\delta g(t)^\beta\,\dlat t}{\int_{\gamma}^\delta g(t)^\beta\,\dlat t}
\geq\frac{\int_{g_0}^\delta g(t)^\beta\,\dlat t}{\int_{\gamma}^\delta g(t)^\beta\,\dlat t}
=\left(\frac{\delta-g_0}{\delta-\gamma}\right)^{\beta+1}\\
&=\left(1-\frac{\alpha-\beta+1}{\alpha+2}\right)^{\beta+1}=\left(\frac{\beta+1}{\alpha+2}\right)^{\beta+1},
\end{split}
\end{equation*}
as desired. This finishes the proof of \eqref{e:int_ineq_beta<alpha}.

\subsection{The case of  $\alpha<\beta$}\label{s:proof_alpha<beta}
Now we show the second part of Theorem \ref{t:functional_centroid_general_alpha_beta}, namely, \eqref{e:int_ineq_alpha<beta} provided that $\alpha<\beta$.
We point out that here we use an approach similar to the one followed in Subsection \ref{s:proof_beta<alpha}, but with the main difference that we need to truncate the sets $L$ and $L_g$ due to certain integrability issues, since now the exponent of some functions under the integral sign (vanishing at some points of the domains of integration) is $\alpha-\beta<0$.

\smallskip

We start by considering the function $W_1:[0,\|h\|_\infty]\longrightarrow[0,\mu_\beta(L)]$ given by
\[
W_1(s)=\mu_\beta\bigl(\{x\in L\,:\,\langle x,\e_2\rangle\leq s\}\bigr),
\]
which is clearly both strictly increasing and surjective.
We may then define the function $w_1:[0,\|h\|_\infty]\longrightarrow[\gamma,\delta]$ that satisfies 
\[
W_1(s)=\mu_\beta\bigl(\{x\in L_g\,:\,\langle x,\e_1\rangle\geq w_1(s)\}\bigr)
\]
for any $s\in[0,\|h\|_\infty]$.
Indeed, since
\[
W_1(s)=\frac{c^\beta\bigl(\delta-w_1(s)\bigr)^{\beta+1}}{\beta(\beta+1)},
\]
we get that
\[
w_1(s)=\delta-\left(\frac{\beta(\beta+1)}{c^\beta}W_1(s)\right)^{1/(\beta+1)}.
\]
Note that, from the Borell-Brascamp-Lieb inequality \eqref{e:BBL}, we have that
the function $W_1^{1/(\beta+1)}$ is concave (since the density of $\mu_\beta$, with respect to the Lebesgue measure, is $\bigl(1/(\beta-1)\bigr)$-concave).
Therefore, $w_1$ is strictly decreasing, surjective and convex, and then there exists the function $\varphi_1=w_1^{-1}:[\gamma,\delta]\longrightarrow[0,\|h\|_\infty]$, which is further (strictly decreasing and) concave.

\smallskip

Now, for any $0<\varepsilon\leq\|h\|_\infty$ we define the sets
\begin{equation*}
L_\varepsilon:=\{x\in L\,:\langle x,\e_2\rangle\geq\varepsilon\}
\end{equation*}
and
\begin{equation*}
L_{g,\varepsilon}:=\{x\in L_g\,:\langle x,\e_1\rangle\leq w_1(\varepsilon)\}.
\end{equation*}

\smallskip

\begin{figure}[h]
\centering
\includegraphics[width=\textwidth]{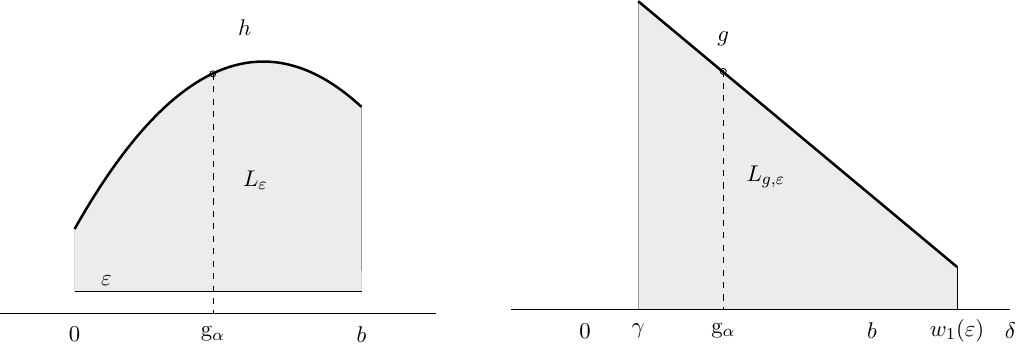}
\caption{Sets $L_{\varepsilon}$ and $L_{g,\varepsilon}$.}
\label{f:Lvarep_and_L_gvarep}
\end{figure}

Notice that, from the definition of $W_1$ and $w_1$ (jointly with (i) in \eqref{e:conditions_decreasing_affine_II}) we have that
\begin{equation}\label{e:conditions_Leps_Lgeps}
\begin{array}{l}
\text{(i)}\; \mu_\beta(L_\varepsilon)=\mu_\beta(L_{g,\varepsilon}), \quad \text{and}\\[4mm]
\text{(ii)}\; \mu_\beta\bigl(\{x\in L_\varepsilon\,:\langle x,\e_2\rangle\leq s\}\bigr) \!=\!
\mu_\beta\bigl(\{x\in L_{g,\varepsilon}\,:\langle x,\e_1\rangle\geq w_1(s)\}\bigr),\\[4mm]
\text{for all } \varepsilon\leq s\leq\|h\|_\infty.
\end{array}
\end{equation}

\smallskip

We will first prove that, for any $0<\varepsilon\leq\|h\|_\infty$, we have
\begin{equation}\label{e:bounding_centroide_alpha(h)_varepsilon}
\begin{split}
\frac{\int_{L_\varepsilon}\langle x,\e_1\rangle\langle x,\e_2\rangle^{\alpha-\beta}\,\dlat\mu_\beta(x)}
{\int_{L_\varepsilon}\langle x,\e_2\rangle^{\alpha-\beta}\,\dlat\mu_\beta(x)}
&<\frac{\int_{L_{g,\varepsilon}}\langle x,\e_1\rangle\varphi_1(\langle x,\e_1\rangle)^{\alpha-\beta}\,\dlat\mu_\beta(x)}
{\int_{L_{g,\varepsilon}}\varphi_1(\langle x,\e_1\rangle)^{\alpha-\beta}\,\dlat\mu_\beta(x)}\\[2mm]
&+\frac{b\delta\varepsilon^\alpha}{\int_{L_\varepsilon}\langle x,\e_2\rangle^\alpha \,\dlat x}.
\end{split}
\end{equation}

To this aim, we will observe that, denoting by $b_\varepsilon=\max\bigl\{x\in\R: (x,y)\in L_\varepsilon\bigr\}$,
for any $0\leq s\leq b_\varepsilon$ we have that
\begin{equation}\label{e:mu_beta_Leps+(e_1,s)<W_1(eps)+mu_beta_Lg,eps+(e_1,s)}
\mu_\beta\bigl(\{x\in L_\varepsilon\,:\,\langle x,\e_1\rangle\geq s\}\bigr)
\leq \mu_\beta\bigl(\{x\in L_{g,\varepsilon}\,:\,\langle x,\e_1\rangle\geq s\}\bigr)+W_1(\varepsilon).
\end{equation}
Indeed, taking into account that if $s\leq w_1(\varepsilon)$ then
\begin{equation*}
\mu_\beta\bigl(\{x\in L_g\,:\,\langle x,\e_1\rangle\geq s\}\bigr)
=\mu_\beta\bigl(\{x\in L_{g,\varepsilon}\,:\,\langle x,\e_1\rangle\geq s\}\bigr)+W_1(\varepsilon),
\end{equation*}
and if $s>w_1(\varepsilon)$ then
\begin{equation*}
\mu_\beta\bigl(\{x\in L_g\,:\,\langle x,\e_1\rangle\geq s\}\bigr)\leq W_1(\varepsilon),
\end{equation*}
we get, from \eqref{e:ineq_beta_measures}, that
\begin{equation*}
\begin{split}
\mu_\beta\bigl(\{x\in L_\varepsilon\,:\,\langle x,\e_1\rangle\geq s\}\bigr)
&\leq\mu_\beta\bigl(\{x\in L\,:\,\langle x,\e_1\rangle\geq s\}\bigr)\\
&\leq\mu_\beta\bigl(\{x\in L_g\,:\,\langle x,\e_1\rangle\geq s\}\bigr)\\
&\leq\mu_\beta\bigl(\{x\in L_{g,\varepsilon}\,:\,\langle x,\e_1\rangle\geq s\}\bigr)+W_1(\varepsilon)
\end{split}
\end{equation*}
for all $0\leq s\leq b_\varepsilon$, which shows \eqref{e:mu_beta_Leps+(e_1,s)<W_1(eps)+mu_beta_Lg,eps+(e_1,s)}.
This, together with
\[
W_1(\varepsilon)\leq\int_0^b\int_0^\varepsilon y^{\beta-1}\,\dlat y\,\dlat x = \frac{b}{\beta}\varepsilon^\beta
\]
and (ii) in \eqref{e:conditions_Leps_Lgeps},
implies that
\begin{equation*}
\begin{split}
\min &\left\{\mu_\beta\bigl(L_\varepsilon^+(\e_1,s_1)\bigr),\mu_\beta\bigl(L_\varepsilon^-(\e_2,1/s_2)\bigr)\right\}\\
&\leq\min\left\{\mu_\beta\bigl(L_{g,\varepsilon}^+(\e_1,s_1)\bigr)+
W_1(\varepsilon),\mu_\beta\Bigl(L_{g,\varepsilon}^+\bigl(\e_1, w_1(1/s_2)\bigr)\Bigr)\right\}\\
&\leq W_1(\varepsilon)+ \min\left\{\mu_\beta\bigl(L_{g,\varepsilon}^+(\e_1,s_1)\bigr),\mu_\beta\Bigl(L_{g,\varepsilon}^+\bigl(\e_1, w_1(1/s_2)\bigr)\Bigr)\right\}\\
&\leq \frac{b}{\beta}\varepsilon^\beta+ \min\left\{\mu_\beta\bigl(L_{g,\varepsilon}^+(\e_1,s_1)\bigr),\mu_\beta\Bigl(L_{g,\varepsilon}^+\bigl(\e_1, w_1(1/s_2)\bigr)\Bigr)\right\}
\end{split}
\end{equation*}
for all $0\leq s_1\leq b_\varepsilon$ and all $0<1/s_2\leq\|h\|_\infty$, where again the notations $L_{\varepsilon}^+(u,c)$ and $L_{g,\varepsilon}^+(u,c)$ represent the sets 
$\bigl(L_\varepsilon\bigr)^+(u,c)$ and $\bigl(L_{g,\varepsilon}\bigr)^+(u,c)$, respectively,  for a given direction $u\in\s^{n-1}$ and a real number $c\in\R$.

\smallskip

So, defining $w_1(s):=\gamma$ if $s\geq\|h\|_\infty$, and taking into account that $b_\varepsilon\leq b\leq \delta$ (by (iii) in \eqref{e:conditions_decreasing_affine_II}), from the fact that $\alpha<\beta$ and using Fubini's theorem we have on the one hand that
\begin{equation}\label{e:upper_bound_numer_centroide_alpha(h)_varepsilon}
\begin{split}
& \!\frac{1}{\beta-\alpha}\int_{L_\varepsilon}\langle x,\e_1\rangle\langle x,\e_2\rangle^{\alpha-\beta} \,\dlat\mu_\beta(x)\\
&\quad\,=\int_{L_\varepsilon}\int_0^{\langle x,\e_1\rangle}\dlat s_1\int_0^{1/\langle x,\e_2\rangle}s_2^{\beta-\alpha-1}\,\dlat s_2\,\dlat\mu_\beta(x)\\
&\quad\,=\int_0^{b_\varepsilon}\int_0^{1/\varepsilon}s_2^{\beta-\alpha-1}\mu_\beta\bigl(L_\varepsilon^+(\e_1,s_1)\cap L_\varepsilon^-(\e_2,1/s_2)\bigr)\,\dlat s_2\,\dlat s_1\\
&\quad\,<\int_0^{b_\varepsilon}\int_0^{1/\varepsilon}s_2^{\beta-\alpha-1}\min\left\{
\mu_\beta\bigl(L_\varepsilon^+(\e_1,s_1)\bigr),\mu_\beta\bigl(L_\varepsilon^-(\e_2,1/s_2)\bigr)\right\}
\,\dlat s_2 \, \dlat s_1\\
&\quad\,\leq\int_0^{\delta}\int_0^{1/\varepsilon}s_2^{\beta-\alpha-1}\frac{b}{\beta}\varepsilon^\beta
\,\dlat s_2\,\dlat s_1\\
&\quad\,+\int_0^{\delta}\int_0^{1/\varepsilon}
s_2^{\beta-\alpha-1}\min\left\{\mu_\beta\bigl(L_{g,\varepsilon}^+(\e_1,s_1)\bigr),
\mu_\beta\Bigl(L_{g,\varepsilon}^+\bigl(\e_1, w_1(1/s_2)\bigr) \Bigr)\right\}\,\dlat s_2\,\dlat s_1\\
&\quad\,=\frac{b\delta}{\beta-\alpha}\varepsilon^\alpha+\int_0^{\delta}\int_0^{1/\varepsilon}s_2^{\beta-\alpha-1}
\mu_\beta\Bigl(L_{g,\varepsilon}^+(\e_1,s_1)\cap L_{g,\varepsilon}^+\bigl(\e_1, w_1(1/s_2)\bigr)\Bigr)\,\dlat s_2\,\dlat s_1\\
&\quad\,=\frac{b\delta}{\beta-\alpha}\varepsilon^\alpha+\frac{1}{\beta-\alpha}\int_{L_{g,\varepsilon}}\langle x,\e_1\rangle\varphi_1(\langle x,\e_1\rangle)^{\alpha-\beta}\,\dlat\mu_\beta(x),
\end{split}
\end{equation}
where in the last equality above we have also used that $\langle x,\e_1\rangle\geq w_1(1/s_2)$ if and only if $1/\varphi_1(\langle x,\e_1\rangle)$ $\geq s_2$ (since $\varphi_1$ is decreasing).

\smallskip

On the other hand, since $\alpha<\beta$, from Fubini's theorem (jointly with (ii) in \eqref{e:conditions_Leps_Lgeps})
we have that
\begin{equation}\label{e:denom_centroide_alpha(h)_varepsilon}
\begin{split}
\frac{1}{\beta-\alpha}\int_{L_\varepsilon}\langle x,\e_2\rangle^{\alpha-\beta}\,\dlat\mu_\beta(x)
&=\int_{L_\varepsilon}\int_0^{1/\langle x,\e_2\rangle}s^{\beta-\alpha-1}\,\dlat s\,\dlat\mu_\beta(x)\\
&=\int_0^{1/\varepsilon}s^{\beta-\alpha-1}\mu_\beta\bigl(L_{\varepsilon}^-(\e_2,1/s)\bigr)\,\dlat s\\
&=\int_0^{1/\varepsilon}s^{\beta-\alpha-1}\mu_\beta\Bigl(L_{g,\varepsilon}^+\bigl(\e_1,w_1(1/s)\bigr)\Bigr)\,\dlat s\\
&=\frac{1}{\beta-\alpha}\int_{L_{g,\varepsilon}}\varphi_1(\langle x,\e_1\rangle)^{\alpha-\beta}\,\dlat\mu_\beta(x),
\end{split}
\end{equation}
where in the last equality above we have used again that $\langle x,\e_1\rangle\geq w_1(1/s)$ if and only if $1/\varphi_1(\langle x,\e_1\rangle)$ $\geq s$.
Hence, from \eqref{e:upper_bound_numer_centroide_alpha(h)_varepsilon} and \eqref{e:denom_centroide_alpha(h)_varepsilon}, we obtain \eqref{e:bounding_centroide_alpha(h)_varepsilon}, as desired.

\medskip

Now we will prove that for any concave function $\varphi_1:[\gamma,\delta]\longrightarrow[0,\infty)$ we have that
\begin{equation}\label{e:ineq_L_g_varepsilon_varphi}
\frac{\int_{L_{g,\varepsilon}}\langle x,\e_1\rangle\varphi_1(\langle x,\e_1\rangle)^{\alpha-\beta}\,\dlat\mu_\beta(x)}{\int_{L_{g,\varepsilon}}\varphi_1(\langle x,\e_1\rangle)^{\alpha-\beta}\,\dlat\mu_\beta(x)}
\leq \frac{\int_{L_{g,\varepsilon}}\langle x,\e_1\rangle\left(\delta-\langle x,\e_1\rangle\right)^{\alpha-\beta} \,\dlat\mu_\beta(x)}{\int_{L_{g,\varepsilon}}\left(\delta-\langle x,\e_1\rangle\right)^{\alpha-\beta}\, \dlat\mu_\beta(x)}.
\end{equation}
To this aim, let $C_1>0$ be such that
\begin{equation}\label{e:C_1_L_g_varepsilon_varphi}
\int_{L_{g,\varepsilon}}\bigl(C_1(\varepsilon)(\delta-\langle x,\e_1\rangle)\bigr)^{\alpha-\beta}\,\dlat\mu_\beta(x)
=\int_{L_{g,\varepsilon}}\varphi_1(\langle x,\e_1\rangle)^{\alpha-\beta}\,\dlat\mu_\beta(x).
\end{equation}
Since the latter identity is equivalent (by Fubini's theorem) to
\begin{equation*}
\int_\gamma^{w_1(\varepsilon)}\left(\bigl(C_1(\varepsilon)(\delta-t)\bigr)^{\alpha-\beta}
-\varphi_1(t)^{\alpha-\beta}\right)
g^\beta(t)\,\dlat t=0,
\end{equation*}
we may assert, taking into account that $\varphi_1$ is concave, that
there exists $t_0(\varepsilon)\in(\gamma,\delta)$ such that
\begin{equation}\label{e:C_1_L_g_varepsilon_varphi2}
\begin{array}{l}
\text{(i)}\; C_1(\varepsilon)(\delta-t)\geq\varphi_1(t) \text{ for every } \gamma\leq t\leq t_0(\varepsilon), \quad \text{and}\\[6mm]
\text{(ii)}\; C_1(\varepsilon)(\delta-t)\leq\varphi_1(t) \text{ for every } t_0(\varepsilon)\leq t\leq w_1(\varepsilon).
\end{array}
\end{equation}

Then, from \eqref{e:C_1_L_g_varepsilon_varphi} and \eqref{e:C_1_L_g_varepsilon_varphi2} (taking into account that $\alpha<\beta$), and using Fubini's theorem, we obtain
\begin{equation*}
\begin{split}
\int_{L_{g,\varepsilon}}\langle x,\e_1\rangle&\varphi_1(\langle x,\e_1\rangle)^{\alpha-\beta}\,\dlat\mu_\beta(x)
-\int_{L_{g,\varepsilon}}\langle x,\e_1\rangle\bigl(C_1(\varepsilon)(\delta-
\langle x,\e_1\rangle)\bigr)^{\alpha-\beta}\,\dlat\mu_\beta(x)\\
=&\,\frac{1}{\beta}\int_\gamma^{w_1(\varepsilon)} t\Bigl(\varphi_1(t)^{\alpha-\beta}-\bigl(C_1(\varepsilon)(\delta-t)\bigr)^{\alpha-\beta}\Bigr)g^\beta(t)\,\dlat t\\
=&\,\frac{1}{\beta}\int_\gamma^{t_0(\varepsilon)} t\Bigl(\varphi_1(t)^{\alpha-\beta}-\bigl(C_1(\varepsilon)(\delta-t)\bigr)^{\alpha-\beta}\Bigr)g^\beta(t)\,\dlat t\\
+&\,\frac{1}{\beta}\int_{t_0}^{w_1(\varepsilon)} t\Bigl(\varphi_1(t)^{\alpha-\beta}-\bigl(C_1(\varepsilon)(\delta-t)\bigr)^{\alpha-\beta}\Bigr)g^\beta(t)\,\dlat t\\
\leq&\,\frac{t_0(\varepsilon)}{\beta} \int_\gamma^{t_0(\varepsilon)}\left(\varphi_1(t)^{\alpha-\beta}-\bigl(C_1(\varepsilon)(\delta-t)\bigr)^{\alpha-\beta}\right)
g^\beta(t)\,\dlat t\\
+&\,\frac{t_0(\varepsilon)}{\beta}\int_{t_0(\varepsilon)}^{w_1(\varepsilon)} \left(\varphi_1(t)^{\alpha-\beta}-\bigl(C_1(\varepsilon)(\delta-t)\bigr)^{\alpha-\beta}\right)g^\beta(t)\,\dlat t\\
=&\,\frac{t_0(\varepsilon)}{\beta}\int_\gamma^{w_1(\varepsilon)}\left(\varphi_1(t)^{\alpha-\beta}-\bigl(C_1(\varepsilon)(\delta-t)\bigr)^{\alpha-\beta}\right)
g^\beta(t)\,\dlat t\\
=& \,t_0(\varepsilon)\int_{L_{g,\varepsilon}}\left(\varphi_1(\langle x,\e_1\rangle)^{\alpha-\beta}-
\bigl(C_1(\varepsilon)(\delta-\langle x,\e_1\rangle)\bigr)^{\alpha-\beta}\right)\,\dlat\mu_\beta(x)=0.
\end{split}
\end{equation*}
Thus, we have
\begin{equation*}
\int_{L_{g,\varepsilon}}\langle x,\e_1\rangle\varphi_1(\langle x,\e_1\rangle)^{\alpha-\beta}\,\dlat\mu_\beta(x)
\leq \int_{L_{g,\varepsilon}}\langle x,\e_1\rangle\bigl(C_1(\varepsilon)
(\delta-\langle x,\e_1\rangle)\bigr)^{\alpha-\beta} \,\dlat\mu_\beta(x),
\end{equation*}
which, together with \eqref{e:C_1_L_g_varepsilon_varphi}, yields \eqref{e:ineq_L_g_varepsilon_varphi}.

\smallskip

Hence, from \eqref{e:bounding_centroide_alpha(h)_varepsilon} and \eqref{e:ineq_L_g_varepsilon_varphi}, for every $0<\varepsilon\leq\|h\|_\infty$ we get
\begin{equation*}
\begin{split}
\frac{\int_{L_\varepsilon}\langle x,\e_1\rangle\langle x,\e_2\rangle^{\alpha-\beta}\,\dlat\mu_\beta(x)}
{\int_{L_\varepsilon}\langle x,\e_2\rangle^{\alpha-\beta}\,\dlat\mu_\beta(x)}
&<\frac{\int_{L_{g,\varepsilon}}\langle x,\e_1\rangle\left(\delta-\langle x,\e_1\rangle\right)^{\alpha-\beta} \,\dlat\mu_\beta(x)}
{\int_{L_{g,\varepsilon}}\left(\delta-\langle x,\e_1\rangle\right)^{\alpha-\beta} \,\dlat\mu_\beta(x)}\\[2mm]
&+\frac{b\delta\varepsilon^\alpha}{\int_{L_\varepsilon}\langle x,\e_2\rangle^\alpha\,\dlat x}.
\end{split}
\end{equation*}

Now, taking limits as $\varepsilon\to0^+$ in the above inequality, we have that the left-hand side, namely,
\[
\frac{\int_{L_\varepsilon}\langle x,\e_1\rangle\langle x,\e_2\rangle^{\alpha-\beta}\,\dlat\mu_\beta(x)}
{\int_{L_\varepsilon}\langle x,\e_2\rangle^{\alpha-\beta}\,\dlat\mu_\beta(x)}
=\frac{\int_{L_\varepsilon}\langle x,\e_1\rangle\langle x,\e_2\rangle^{\alpha-1}\,\dlat x}
{\int_{L_\varepsilon}\langle x,\e_2\rangle^{\alpha-1}\,\dlat x},
\]
tends to
\[
\frac{\int_{L}\langle x,\e_1\rangle\langle x,\e_2\rangle^{\alpha-1}\,\dlat x}
{\int_{L}\langle x,\e_2\rangle^{\alpha-1}\,\dlat x}=\centroid{\alpha}(h)
\]
as $\varepsilon\to0^+$ (see \eqref{e:centroide_alpha(h)}).
Furthermore, the first term in the right-hand side,
\begin{equation*}
\frac{\int_{L_{g,\varepsilon}}\langle x,\e_1\rangle\left(\delta-\langle x,\e_1\rangle\right)^{\alpha-\beta} \,\dlat\mu_\beta(x)}{\int_{L_{g,\varepsilon}}\left(\delta-\langle x,\e_1\rangle\right)^{\alpha-\beta} \,\dlat\mu_\beta(x)}
=\frac{\frac{c^\beta}{\beta}\int_\gamma^{w_1(\varepsilon)}t(\delta-t)^{\alpha}\,\dlat t} {\frac{c^\beta}{\beta}\int_\gamma^{w_1(\varepsilon)}(\delta-t)^{\alpha}\,\dlat t},
\end{equation*}
tends to
\[
\centroid{\alpha}(g)=\frac{\int_\gamma^{\delta}t(\delta-t)^{\alpha}\,\dlat t} {\int_\gamma^{\delta}(\delta-t)^{\alpha}dt}
=\left(\delta -
\frac{(\alpha+1)(\delta-\gamma)}{\alpha+2}\right)
\]
for $\varepsilon\to0^+$, whereas
the second term in the right-hand side,
\[
\frac{b\delta\varepsilon^\alpha}{\int_{L_\varepsilon}\langle x,\e_2\rangle^\alpha\, \dlat x},
\]
clearly tends to
\[\frac{0}{\int_{L}\langle x,\e_2\rangle^\alpha\, \dlat x}=0
\]
for $\varepsilon\to0^+$.
Therefore, we have that
\[
\centroid{\alpha}(h)\leq \centroid{\alpha}(g).
\]
Finally, the latter relation jointly with (ii) and (iii) in \eqref{e:conditions_decreasing_affine} gives us
\begin{equation*}
\begin{split}
\frac{\int_{\centroid{\alpha}(h)}^b h(t)^\beta\,\dlat t}
{\int_{0}^b h(t)^\beta\,\dlat t}
&=\frac{\int_{\centroid{\alpha}(h)}^\delta g(t)^\beta\,\dlat t}
{\int_{\gamma}^\delta g(t)^\beta\,\dlat t}
\geq\frac{\int_{\centroid{\alpha}(g)}^\delta g(t)^\beta\,\dlat t}
{\int_{\gamma}^\delta g(t)^\beta\,\dlat t}=\left(\frac{\delta-\centroid{\alpha}(g)}{\delta-\gamma}\right)^{\beta+1}\\
&=\left(\frac{\alpha+1}{\alpha+2}\right)^{\beta+1},
\end{split}
\end{equation*}
as desired. This finishes the proof of \eqref{e:int_ineq_alpha<beta}, and hence also that of Theorem \ref{t:functional_centroid_general_alpha_beta}.

\end{document}